\documentclass[12pt]{amsart}
\usepackage{amssymb}
\usepackage{mathrsfs}
\usepackage{amssymb}
\usepackage{amsfonts}
\usepackage[english]{babel}
\usepackage[T1]{fontenc}
\usepackage[latin1]{inputenc}
\usepackage{fullpage}
\usepackage{amsmath}
\usepackage[all]{xy}
\usepackage{stmaryrd}

\newtheorem{theorem}{Theorem}[section]
\newtheorem{lemma}[theorem]{Lemma}
\newtheorem{corollary}[theorem]{Corollary}
\theoremstyle{definition}
\newtheorem{definition}[theorem]{Definition}

\newtheorem{proposition}[theorem]{Proposition}

\theoremstyle{remark}
\newtheorem{remark}[theorem]{Remark}

\numberwithin{equation}{section}

\begin{document}

\title{A symmetric $2$-tensor canonically associated to $Q$-curvature and its applications}


\author{Yueh-Ju Lin}
\address{(Yueh-Ju Lin) MSRI and Department of Mathematics, University of Michgan, Ann Arbor} 
\curraddr{The Mathematical Sciences Research Institute, Berkeley, CA 94720, USA}
\email{yuehjul@umich.edu}

\author{Wei Yuan}
\address{(Wei Yuan) School of Mathematics and Computational Science, Sun Yat-sen University, Guangzhou, Guangdong 510275, China}
\email{gnr-x@163.com}





\thanks{}

\begin{abstract}
In this article, we define a symmetric $2$-tensor canonically associated to $Q$-curvature called $J$-tensor on any Riemannian manifold with dimension at least three. The relation between $J$-tensor and $Q$-curvature is precisely like Ricci tensor and scalar curvature. Thus it can be interpreted as a higher-order analogue of Ricci tensor. This tensor can also be used to understand Chang-Gursky-Yang's theorem on $4$-dimensional $Q$-singular metrics. Moreover, we show an \emph{Almost-Schur Lemma} holds for $Q$-curvature, which gives an estimate of $Q$-curvature on closed manifolds. 
\end{abstract}

\maketitle



\section{Introduction}

Let $M$ be a smooth manifold and $\mathcal{M}$ be the space of all metrics on $M$. Consider scalar curvature as a nonlinear map $$R: \mathcal{M} \rightarrow C^{\infty}(M); \ g\mapsto R_g.$$

It is well-known that the linearization of scalar curvature at a given metric $g$ (see \cite{Besse, C-L-N, F-M}) is
\begin{align}
\gamma_g h := DR_g\cdot h = - \Delta_g tr_g h + \delta_g^2 h - Ric_g \cdot h,
\end{align}
where $h\in S_2(M)$ is a symmetric $2$-tensor and $\delta_g = - div_g$. Thus, its $L^2$-formal adjoint is given by
\begin{align}
\gamma_g^* f = \nabla_g^2 f - g \Delta_g f - f Ric_g,
\end{align}
for any smooth function $f \in C^\infty(M)$.\\

An interesting observation is that, if we take $f$ to be constantly $1$, we get
$$ Ric_g = -\gamma_g^* 1 $$
That means we can recover Ricci tensor from $\gamma_g^*$. Furthermore, the scalar curvature is given by $$R_g= - tr_g \gamma_g^* 1.$$

Now let $(M^n,g)$ be an $n$-dimensional Riemannian manifold ($n\geq 3$). We can define the $Q$-curvature to be 
\begin{align}
Q_{g} = A_n \Delta_{g} R_{g} + B_n |Ric_{g}|_{g}^2 + C_nR_{g}^2,
\end{align}
where $A_n = - \frac{1}{2(n-1)}$ , $B_n = - \frac{2}{(n-2)^2}$ and
$C_n = \frac{n^2(n-4) + 16 (n-1)}{8(n-1)^2(n-2)^2}$.

In fact, $Q$-curvature was introduced originally to generalize the classic \emph{Gauss-Bonnet Theorem} on surfaces to closed $4$-manifolds $(M^4, g)$:
\begin{align}\label{Gauss_Bonnet_Chern}
\int_{M^4} \left( Q_g + \frac{1}{4} |W_g|^2_g \right) dv_g = 8\pi^2 \chi(M).
\end{align}
where $W_g$ is the Weyl tensor. 

Paneitz and Branson extended it to any dimension $n \geq 3$ (cf. \cite{Branson, Paneitz}) such that it satisfies certain conformal invariant properties. For more details, please refer to the appendix of \cite{L-Y}.\\

Like the scalar curvature, we can also view $Q$-curvature as a nonlinear map
$$Q: \mathcal{M} \rightarrow C^{\infty}(M); \ g\mapsto Q_g.$$

Let $$\Gamma_g : S_2(M) \rightarrow C^\infty(M)$$ be the linearization of $Q$-curvature at the metric $g$ and $$\Gamma_g^* : C^\infty(M) \rightarrow S_2(M)$$ be its $L^2$-formal adjoint.\\

Now we can define the central notion in this article:
\begin{definition}
Let $(M^n,g)$ be a Riemannian manifold ($n\geq3$). We define the symmetric $2$-tensor $$J_g : = -\frac{1}{2} \Gamma_g^* 1.$$ We say $(M,g)$ is \emph{$J$-Einstein}, if $J_g = \Lambda g$ for some smooth function $\Lambda \in C^\infty(M)$. In particular, it is \emph{$J$-flat}, if $\Lambda = 0$.
\end{definition}
 
 In \cite{L-Y}, we calculated the explicit expression of $\Gamma_g^*$ and showed that 
\begin{align}
 tr_g \Gamma_g^* f = \frac{1}{2} \left( P_g - \frac{n+4}{2} Q_g\right) f,
\end{align} 
for any $f \in C^\infty(M)$. Here $P_g$ is the \emph{Paneitz operator} defined by
\begin{align}
P_g = \Delta_g^2 - div_g \left[(a_n R_g g + b_n Ric_g) d\right] + \frac{n-4}{2}Q_g,
\end{align}
where $a_n = \frac{(n-2)^2 + 4}{2(n-1)(n-2)}$ and $b_n = - \frac{4}{n-2}$.

In particular, 
\begin{align*}
tr_g \Gamma_g^* 1  = - 2Q_g.
\end{align*}

Thus 
\begin{align}
tr_g J_g = Q_g.
\end{align}

On the other hand, for any smooth vector field $X \in \mathscr{X}(M)$ on $M$, 
\begin{align*}
\int_M \langle X,   \delta_g \Gamma_g^* f \rangle dv_g = \frac{1}{2}\int_M \langle L_X g, \Gamma_g^* f \rangle dv_g =\frac{1}{2} \int_M f\ \Gamma_g (L_X g)\ dv_g  = \frac{1}{2}\int_M \langle f dQ_g, X \rangle \ dv_g.
\end{align*}
Thus $$\delta_g \Gamma_g^* f =\frac{1}{2} fdQ_g$$ on $M$. Hence, 
\begin{align}
div_g J_g = \frac{1}{2} \delta_g \Gamma_g^* 1 = \frac{1}{4}dQ_g.
\end{align}

Recall that for Ricci tensor, we have $$tr_g Ric_g = R_g$$ and $$div_g Ric_g = \frac{1}{2} dR_g.$$

Therefore, if we consider $Q$-curvature as a higher-order analogue of scalar curvature, we can interpret $J_g$ as a higher-order analogue of Ricci curvature on Riemannian manifolds.\\

A notion closely related to $J$-tensor is the \emph{$Q$-singular metric}, which refers to a metric satisfying $\ker \Gamma_g^* \neq \{0\}$. Clearly, $J$-flat metrics are $Q$-singular, since it is equivalent to $1 \in \ker \Gamma_g^*$.

One of the motivations for us to study the $J$-flat manifold is to understand the following theorem by Chang-Gursky-Yang:
\begin{theorem}[Chang-Gursky-Yang \cite{C-G-Y}]\label{thm:C-G-Y}
Let $(M^4, g)$ be a $Q$-singular $4$-manifold. Then $1 \in \ker \Gamma_g^*$ if and only if $(M^4,g)$ is Bach flat with vanishing $Q$-curvature.
\end{theorem}

To achieve our goal, we need to give the explicit expression of $J$-tensor:
\begin{theorem}\label{thm-J_expression}
For $n\geq 3$,
\begin{align}\label{eqn:J_expression}
J_g = \frac{1}{n} Q_g g  - \frac{1}{n-2} B_g   - \frac{n-4}{4(n-1)(n-2)} T_g,
\end{align}
where $B_g$ is the Bach tensor and
\begin{align*}
T_g:= (n-2)\left( \nabla^2 tr_g S_g - \frac{1}{n} g \Delta_g tr_g S_g\right) + 4(n-1)\left( S_g\times S_g - \frac{1}{n} |S_g|^2 g \right) - n^2 (tr_g S_g) \overset{\circ}{S}_g.
\end{align*}
Here $( S\times S)_{jk} = S_j^iS_{ik}$, $S_g$ is the Schouten tensor and $\overset{\circ}{S}_g$ is its traceless part. 
\end{theorem}

\begin{remark}
Note that both the Bach tensor and the tensor $T$ are traceless, thus the traceless part of $J$ is given by
 \begin{align}
\overset{\circ}{J}_g = J_g - \frac{1}{n} Q_g g = - \frac{1}{n-2} \left( B_g   + \frac{n-4}{4(n-1)}T_g\right).
\end{align}
Thus, an equivalent definition for a metric $g$ being $J$-Einstein is
\begin{align}
B_g = -\frac{n-4}{4(n-1)} T_g.
\end{align}
In particular, when $n=4$, $J$-Einstein metrics are exactly Bach flat ones. Hence we can also interpret that $J$-Einstein metric is a generalization of Bach flat metric on $4$-dimensional manifolds.
\end{remark}

\begin{remark}
Gursky introduced a similar tensor for $4$-manifolds from the viewpoint of functional determinant in \cite{Gursky}. In the same article, he also remarked this tensor can be introduced from the perspective of first variations of total $Q$-curvature when dimension is at least 5 (see \cite{Case} for a detailed calculation by Case).  

With the similar perspective, Gover and \O rsted introduced an abstract tensor called \emph{higher Einstein tensor}, which coincides with our $J$-tensor in one of its special case. We refer their article \cite{G-O} for readers who are interested in it.
\end{remark}

Note that for any Einstein metric $g$, its $Q$-curvature is given by
\begin{align*}
Q_g = B_n|Ric_g|^2 + C_n R_g^2 = \left( \frac{1}{n} B_n + C_n \right) R_g^2 = \frac{(n+2)(n-2)}{8n(n-1)^2} R_g^2,
\end{align*}
which is a nonnegative constant and vanishes if and only if $g$ is Ricci flat.

It is easy to check that $T_g = 0$ for any Einstein metric $g$. Combining this with the well-known fact that any Einstein metric is Bach flat, we can easily deduce that any non-flat Einstein metrics are also positive $J$-Einstein and Ricci flat metrics are $J$-flat as well.\\

With the aid of this notion, we can recover and generalize Theorem \ref{thm:C-G-Y} to any dimension $n \geq 3$:
\begin{corollary}\label{cor:J-flat}
Let $(M^n, g)$ be a $Q$-singular $n$-dimensional Riemannian manifold. Then $1 \in \ker \Gamma_g^*$ if and only if $(M^n,g)$ is $J$-flat or equivalently $(M^n,g)$ satisfies 
\begin{align*}
B_g = -\frac{n-4}{4(n-1)} T_g
\end{align*}
with vanishing $Q$-curvature.
\end{corollary}

\begin{remark}
In \cite{C-G-Y}, Bach flatness in Theorem \ref{thm:C-G-Y} is derived using the variational property of Bach tensor for $4$-manifolds.
\end{remark}

As another application of $J$-tensor, we can derive the \emph{Schur Lemma for $Q$-curvature} as follows:
\begin{theorem}[Schur lemma]\label{thm:Schur_Lemma}
Let $(M^n, g)$ be an $n$-dimensional $J$-Einstein manifold with $n\neq 4$ or equivalently, 
\begin{align*}
B_g = -\frac{n-4}{4(n-1)} T_g,
\end{align*}
then $Q_g$ is a constant on $M$.
\end{theorem}

Moreover, the following \emph{Almost-Schur Lemma} holds exactly like the case for Ricci tensor and scalar curvature (cf. \cite{Cheng, L-T, G-W}). 
\begin{theorem}[Almost-Schur Lemma]\label{thm:almost_schur_lemma}
For $n\neq 4$, let $(M^n,g)$ be an $n$-dimensional closed Riemannian manifold with positive Ricci curvature. Then
\begin{align}
\int_M (Q_g - \overline{Q}_g)^2 dv_g \leq \frac{16n(n-1)}{(n-4)^2}\int_M |\overset{\circ}{J}_g|^2 dv_g,
\end{align}
where $\overline{Q}_g$ is the average of $Q_g$. Moreover, the equality holds if and only if $(M,g)$ is $J$-Einstein.
\end{theorem}

In order to derive an equivalent form of above inequality, we need to define the \emph{$J$-Schouten tensor} as follows,
\begin{align}
S_J = \frac{1}{n-4} \left( J_g - \frac{3}{4(n-1)} Q_g g\right).
\end{align}

Immediately, we have 
\begin{align}
tr_g S_J = \frac{1}{4(n-1)}Q_g
\end{align}
and
\begin{align}
div_g S_J = \frac{1}{4(n-1)}d Q_g = d tr_g S_J.
\end{align}

\begin{remark}
Recall the definition of classic Schouten tensor
\begin{align}
S_g= \frac{1}{n-2} \left( Ric_g - \frac{1}{2(n-1)} R_g g\right),
\end{align}
we have 
\begin{align}
tr_g S_g = \frac{1}{2(n-1)}R_g
\end{align}
and
\begin{align}
div_g S_g = \frac{1}{2(n-1)}d R_g = d tr_g S_g.
\end{align}
We can see that the tensor $S_J$ shares similar properties with the classic Schouten tensor.
\end{remark}

Following the observation in \cite{G-W}, we get immediately the following result by rewriting the Theorem \ref{thm:almost_schur_lemma} with $J$-Schouten tensor:

\begin{corollary}
For $n\neq 4$, let $(M^n,g)$ be an $n$-dimensional closed Riemannian manifold with positive Ricci curvature. Then
\begin{align}
( Vol_g(M) )^{-\frac{n-8}{n}}\int_M \sigma_2^J(g) dv_g \leq \frac{n-1}{2n} Y^2_Q(g),
\end{align}
where
\begin{align*}
Y_Q(g):=\frac{\int_M \sigma_1^J(g) dv_g }{( Vol_g(M) )^{\frac{n-4}{n}}}
\end{align*}
is the \emph{$Q$-Yamabe quotient} and $\sigma_i^J(g) = \sigma_i(S_J(g))$, $i = 1,2$ are the $i^{th}$-symmetric polynomial of $S_J(g)$. Moreover, the equality holds if and only if $(M,g)$ is $J$-Einstein.

\end{corollary}

\begin{remark}
The above \emph{Almost Schur Lemma} can be easily generalized to a broader setting by combining the work \cite{G-O}. More detailed discussions together with some related topics will be presented in a subsequent article coming later.
\end{remark}

This article is organized as follows: in Section 2, we derived the explicit formula for $J$-tensor and with the aid of it we proved Theorem \ref{thm-J_expression} and Corollary \ref{cor:J-flat}; We then proved Theorem \ref{thm:Schur_Lemma} (Schur Lemma) and Theorem \ref{thm:almost_schur_lemma} (Almost-Schur Lemma) in Section 3.\\

\paragraph{\textbf{Acknowledgement}}
We would like to thank Professor Sun-Yung Alice Chang, Professor Matthew Gursky and Professor Jeffery Case for their interests in this work and inspiring discussions. Especially, we would like to thank Professor Jeffery Case for bringing our attentions to the work \cite{G-O, Gursky} and his wonderful comments.

Part of the work was done when the first author was in residence at the Mathematical Sciences Research Institute in Berkeley, supported by the NSF grant DMS-1440140 during Spring 2016. Also, part of the work was done when the second author visited Institut Henri Poincar\'e in Paris during Fall 2015. We would like to express our deepest appreciations to both MSRI and IHP for their sponsorship and hospitality.\\


\section{$J$-flatness and $Q$-singular metrics}

We begin with some discussions of conformal tensors.

Let 
\begin{align}
S_{jk} = \frac{1}{n-2} \left(R_{jk} - \frac{1}{2(n-1)} R g_{jk}\right)
\end{align}
be the Schouten tensor. 

For $n\geq 4$, the Bach tensor is defined to be
\begin{align}
B_{jk} = \frac{1}{n-3} \nabla^i \nabla^l W_{ijkl} + W_{ijkl}S^{il}.
\end{align}

In order to extend the definition to $n=3$, we introduce the Cotton tensor as follows
\begin{align}
C_{ijk} = \nabla_i S_{jk} - \nabla_j S_{ik}
\end{align}
and it is related to Weyl tensor by the equation
\begin{align}
\nabla^l W_{ijkl} = (n-3) C_{ijk}.
\end{align}

Therefore, for any $n \geq 3$, we can defined the Bach tensor as
\begin{align}
B_{jk} = \nabla^i C_{ijk} + W_{ijkl}S^{il}.
\end{align}

The following identity is well-known for experts, we include calculations here for the convenience of readers. 

\begin{proposition}\label{prop:Bach_alt}
The Bach tensor can be written as
\begin{align}
B_g = \Delta_g S - \nabla^2 tr S + 2 \overset{\circ}{Rm} \cdot S - (n-4) S\times S  - |S|^2 g -  2 (tr S) S,
\end{align}
where $(\overset{\circ}{Rm} \cdot S)_{jk} = R_{ijkl}S^{il}$ and $( S\times S)_{jk} = S_j^iS_{ik}$. Equivalently, 
\begin{align}
B_g =& \Delta_L S - \nabla^2 tr S + n \left( S\times S  - \frac{1}{n}|S|^2 g \right),
\end{align}
where $\Delta_L$ is the Lichnerowicz Laplacian.
\end{proposition}

\begin{proof}
By the \emph{second contracted Bianchi identity},  
\begin{align*}
\nabla^i S_{ik} &= \frac{1}{n-2} \left( \nabla^i R_{ik} - \frac{1}{2(n-1)} \nabla_k R\right)\\
&= \frac{1}{n-2} \left( \frac{1}{2} \nabla_k R - \frac{1}{2(n-1)} \nabla_k R\right)\\
&= \frac{1}{2(n-1)} \nabla_k R\\
&= \nabla_k tr S
\end{align*}

and
\begin{align*}
tr S  = \frac{1}{n-2} \left( R - \frac{n}{2(n-1)} R\right) = \frac{1}{2(n-1)} R,
\end{align*}

we have
\begin{align*}
Ric = (n-2) S + (tr S) g. 
\end{align*}

Using these facts, 
\begin{align*}
\nabla^i C_{ijk} &= \nabla^i ( \nabla_i S_{jk} - \nabla_j S_{ik} )\\
&= \Delta_g S_{jk} - ( \nabla_j \nabla_i S^i_k + R_{ijp}^i S_k^p - R_{ijk}^p S_p^i )\\
&= \Delta_g S_{jk} - \nabla_j \nabla_k tr S -  (Ric \times S)_{jk} + (\overset{\circ}{Rm} \cdot S)_{jk}\\
&= \Delta_g S_{jk} - \nabla_j \nabla_k tr S -  \left(\left((n-2) S + (tr S)g \right) \times S\right)_{jk} + (\overset{\circ}{Rm} \cdot S)_{jk}\\
&= \Delta_g S_{jk} - \nabla_j \nabla_k tr S -  (n-2) \left(S\times S\right)_{jk} -  (tr S)S_{jk}+ (\overset{\circ}{Rm} \cdot S)_{jk}\\
\end{align*}
and
\begin{align*}
W_{ijkl}S^{il} &= \left( Rm - S\owedge g\right)_{ijkl} S^{il}\\
&= R_{ijkl} S^{il} - (S_{il} g_{jk} +S_{jk} g_{il} -  S_{ik} g_{jl} - S_{jl} g_{ik}) S^{il}\\
&= (\overset{\circ}{Rm} \cdot S)_{jk} - |S|^2 g_{jk} + 2 (S \times S)_{jk} - (tr S) S_{jk},
\end{align*}
where $\owedge$ is the \emph{Kulkarni-Nomizu product}: $$(\alpha \owedge \beta)_{ijkl} := \alpha_{il} \beta_{jk} +\alpha_{jk} \beta_{il} -  \alpha_{ik} \beta_{jl} - \alpha_{jl} \beta_{ik}$$ for any symmetric $2$-tensor $\alpha, \beta \in S_2(M)$.

Combine them, we get
\begin{align*}
B_{jk} &= \Delta_g S_{jk} - \nabla_j \nabla_k tr S + 2(\overset{\circ}{Rm} \cdot S)_{jk}  - (n-4) (S \times S)_{jk}- |S|^2 g_{jk} - 2(tr S) S_{jk}.
\end{align*}

From this,
\begin{align*}
B_{jk} &= \Delta_L S_{jk} + 2 (Ric \times S)_{jk} - \nabla_j \nabla_k tr S  - (n-4) (S \times S)_{jk}- |S|^2 g_{jk} - 2(tr S) S_{jk}\\
&= \Delta_L S_{jk} + 2 \left(\left(Ric -  (tr S) g\right)\times S \right)_{jk}- \nabla_j \nabla_k tr S  - (n-4) (S \times S)_{jk}- |S|^2 g_{jk}\\
&= \Delta_L S - \nabla^2 tr S + n (S\times S)  - |S|^2 g\\
&= \Delta_L  S - \nabla^2 tr S + n \left( S\times S  - \frac{1}{n}|S|^2 g \right).
\end{align*}
\end{proof}

The $Q$-curvature can also be rewritten using Schouten tensor:
\begin{lemma}\label{lem:Q_alt}
\begin{align}
Q_g = - \Delta_g tr S - 2|S|^2 + \frac{n}{2} (tr S)^2.
\end{align}
\end{lemma}

\begin{proof}
Using the equalities $Ric = (n-2) S + (tr S)g$ and $R = 2(n-1) tr S$,
\begin{align*}
Q_g &= A_n \Delta_g R + B_n |Ric|^2 + C_n R^2\\
&= 2(n-1) A_n \Delta_g tr S + B_n |(n-2) S + (tr S) g|^2 + 4(n-1)^2 C_n (tr S)^2\\
&= - \Delta_g tr S - 2|S|^2 + ((3n -4)B_n + 4(n-1)^2 C_n )(tr S)^2\\
&= - \Delta_g tr S - 2|S|^2 + \frac{n}{2} (tr S)^2.
\end{align*}
\end{proof}

 We recall the expression of $\Gamma_g^*$ in \cite{L-Y} as follows:
 \begin{lemma}\label{lem:Gamma^*}
\begin{align}
\Gamma_g^* f :=& A_n \left(  - g \Delta^2 f + \nabla^2 \Delta f - Ric \Delta f + \frac{1}{2} g \delta (f dR) + \nabla ( f dR) - f \nabla^2 R  \right)\\ \notag
& - B_n \left( \Delta (f Ric) + 2 f\overset{\circ}{Rm}\cdot Ric + g \delta^2 (f Ric) + 2 \nabla \delta (f
Ric) \right)\\ \notag
&- 2 C_n \left( g\Delta (f R) - \nabla^2 (f R) + f R Ric \right).
\end{align}
\end{lemma}
 
Now we can calculate the explicit expression of $J_g$:
\begin{theorem}
For $n\geq 3$,
\begin{align}
J_g = \frac{1}{n} Q_g g  - \frac{1}{n-2} B_g   - \frac{n-4}{4(n-1)(n-2)} T_g,
\end{align}
where 
\begin{align*}
T_g:= (n-2)\left( \nabla^2 tr_g S_g - \frac{1}{n} g \Delta_g tr_g S_g\right) + 4(n-1)\left( S_g\times S_g - \frac{1}{n} |S_g|^2 g \right) - n^2 (tr_g S_g) \overset{\circ}{S}_g.
\end{align*}
Here $\overset{\circ}{S}_g = S_g - \frac{1}{n} tr_g S_g g$ is the traceless part of Schouten tensor.
\end{theorem}

\begin{proof}
By Lemma \ref{lem:Gamma^*},
\begin{align*}
\Gamma_g^* 1 = - &\left(\frac{1}{2}A_n + \frac{1}{2}B_n +  2 C_n \right)g \Delta R + (B_n + 2C_n) \nabla^2 R \\
&- B_n (\Delta Ric + 2 \overset{\circ}{Rm}\cdot Ric  ) - 2 C_n  R Ric. 
\end{align*}

Applying equalities $Ric = (n-2) S + (tr S)g$ and $R = 2(n-1) tr S$,
\begin{align*}
\Gamma_g^* 1 =& - ((n-1)A_n + nB_n +  4(n-1) C_n)g \Delta tr S + 2(n-1)(B_n + 2C_n) \nabla^2 tr S \\
&- (n-2) B_n (\Delta S + 2 \overset{\circ}{Rm}\cdot S  ) - 2(n-2) ( B_n + 2(n-1)C_n ) (tr S) S \\& - 2(B_n + 2(n-1)C_n) (tr S)^2 g\\
=& \frac{3}{2(n-1)} g \Delta tr S + \frac{2}{n-2} (\Delta S + 2 \overset{\circ}{Rm}\cdot S  ) + \frac{n^2 - 10n + 12}{2(n-1)(n-2)}\nabla^2 tr S\\
& - \frac{n^2 - 2n + 4}{2(n-1)}(tr S) S - \frac{n^2 - 2n + 4}{2(n-1)(n-2)}(tr S)^2 g. 
\end{align*}

Since $tr \Gamma_g^* 1 = - 2 Q_g$, by Lemma \ref{lem:Q_alt},
\begin{align*}
\Gamma_g^* 1 + \frac{2}{n}Q_g g
=& \left(\frac{3}{2(n-1)} - \frac{2}{n}\right) g \Delta tr S + \frac{2}{n-2} (\Delta S + 2 \overset{\circ}{Rm}\cdot S  ) + \frac{n^2 - 10n + 12}{2(n-1)(n-2)}\nabla^2 tr S\\
& - \frac{4}{n} |S|^2 g - \frac{n^2 - 2n + 4}{2(n-1)}(tr S) S + \left(1 - \frac{n^2 - 2n + 4}{2(n-1)(n-2)} \right)(tr S)^2 g \\
=&  - \frac{n-4}{2n(n-1)} g \Delta tr S + \frac{2}{n-2} (\Delta S + 2 \overset{\circ}{Rm}\cdot S  ) + \frac{n^2 - 10n + 12}{2(n-1)(n-2)}\nabla^2 tr S\\
& - \frac{4}{n} |S|^2 g - \frac{n^2 - 2n + 4}{2(n-1)}(tr S) S + \frac{n( n- 4)}{2(n-1)(n-2)} (tr S)^2 g. 
\end{align*}

Applying Proposition \ref{prop:Bach_alt},
\begin{align*}
\Gamma_g^* 1 + \frac{2}{n}Q_g g
=& \frac{2}{n-2} B_g - \frac{n-4}{2n(n-1)} g \Delta tr S + \left( \frac{2}{n-2} +\frac{n^2 - 10n + 12}{2(n-1)(n-2)}\right)\nabla^2 tr S\\
& + \frac{2(n-4)}{n-2} S\times S + \left( \frac{2}{n-2}- \frac{4}{n} \right) |S|^2 g+ \left( \frac{4}{n-2} - \frac{n^2 - 2n + 4}{2(n-1)}\right)(tr S) S \\
&+ \frac{n( n- 4)}{2(n-1)(n-2)} (tr S)^2 g.
\end{align*}

That is,
\begin{align*}
\Gamma_g^* 1 + \frac{2}{n}Q_g g
=& \frac{2}{n-2} B_g - \frac{n-4}{2n(n-1)} g \Delta tr S +  \frac{n-4}{2(n-1)}\nabla^2 tr S  + \frac{2(n-4)}{n-2} S\times S \\ &-  \frac{2(n-4)}{n(n-2)}|S|^2 g- \frac{n^2 ( n - 4)}{2(n-1)(n-2)}(tr S) S + \frac{n( n- 4)}{2(n-1)(n-2)} (tr S)^2 g\\
=& \frac{2}{n-2} B_g  + \frac{n-4}{2(n-1)} \left( \nabla^2 tr S - \frac{1}{n} g \Delta tr S \right)+ \frac{2(n-4)}{n-2} \left(S\times S -  \frac{1}{n}|S|^2 g \right)\\ &- \frac{n^2 ( n - 4)}{2(n-1)(n-2)} (tr S) \left( S - \frac{1}{n} (tr S) g \right)\\
=& \frac{2}{n-2} B_g  + \frac{n-4}{2(n-1)(n-2)} T_g,
\end{align*}
where 
\begin{align*}
T_g:= (n-2)\left( \nabla^2 tr_g S_g - \frac{1}{n} g \Delta_g tr_g S_g\right) + 4(n-1)\left( S_g\times S_g - \frac{1}{n} |S_g|^2 g \right) - n^2 (tr_g S_g) \overset{\circ}{S}_g.
\end{align*}

Therefore,
\begin{align*}
J_g = -\frac{1}{2} \Gamma_g^* 1 = \frac{1}{n}Q_g g - \frac{1}{n-2}B_g - \frac{n-4}{4(n-1)(n-2)} T_g.
\end{align*}

\end{proof}

Immediately, we have the following generalization of Theorem \ref{thm:C-G-Y}:
\begin{corollary}
Let $(M^n, g)$ be a $Q$-singular $n$-dimensional Riemannian manifold. Then $1 \in \ker \Gamma_g^*$ if and only if $(M^n,g)$ is $J$-flat or equivalently $(M^n,g)$ satisfies 
\begin{align}
B_g = -\frac{n-4}{4(n-1)} T_g
\end{align}
with vanishing $Q$-curvature.
\end{corollary}

\begin{remark}
Similar result holds for Ricci curvature: a vacuum static space admits a constant static potential if and only if it is Ricci flat (cf. \cite{F-M}).
\end{remark}

\section{An Almost-Schur lemma for $Q$-curvature}

Since the tensor $J_g$ can be interpreted as a higher-order analogue of Ricci tensor, we can also derive the Schur lemma for $J_g$ as follows:
\begin{theorem}[Schur lemma]
Let $(M^n, g)$ be an $n$-dimensional $J$-Einstein manifold with $n\neq 4$ or equivalently,
\begin{align*}
B_g = -\frac{n-4}{4(n-1)} T_g,
\end{align*}
then $Q_g$ is a constant on $M$.
\end{theorem}

\begin{proof}
By the assumption, $J_g = \Lambda g$ for some smooth function $\Lambda$ on $M$. Then $$\Lambda= \frac{1}{n} tr_g J_g = \frac{1}{n} Q_g$$ and $$d\Lambda = div_g J_g = \frac{1}{4}dQ_g.$$

Therefore, $$\frac{n-4}{4n}dQ_g = 0$$ on $M$, which implies that $Q_g$ is a constant on $M$ provided $n\neq 4$.

\end{proof}

\begin{remark}
When $n=4$, $J$-Einstein metrics are exactly Bach flat ones. Due to the conformal invariance of Bach flatness in dimension $4$, we can easily see that the constancy of $Q$-curvature can not always be achieved. Thus the above Schur Lemma does not hold for $4$-dimensional manifolds which is exactly like the classic Schur Lemma for surfaces. 
\end{remark}

In fact,  a more general result can be derived:

\begin{theorem}[Almost-Schur Lemma]
For $n\neq 4$, let $(M^n,g)$ be an $n$-dimensional closed Riemannian manifold with positive Ricci curvature. Then
\begin{align}\label{ineq:almost_Schur}
\int_M (Q_g - \overline{Q}_g)^2 dv_g \leq \frac{16n(n-1)}{(n-4)^2}\int_M |\overset{\circ}{J}_g|^2 dv_g,
\end{align}
where $\overline{Q}_g$ is the average of $Q_g$. Moreover, the equality holds if and only if $(M^n,g)$ is $J$-Einstein.
\end{theorem}

The proof is along the same line as in \cite{L-T}. For completeness, we include it here for the convenience of readers. For more details, please refer to \cite{L-T}.
\begin{proof}
Let $u$ be the unique solution to 
\begin{align*}
\left\{ \aligned  & \Delta_g u = Q_g - \overline{Q}_g,\\ 
& \int_M u dv_g = 0.
\endaligned\right.
\end{align*}

Then 
\begin{align*}
\int_M (Q_g - \overline{Q}_g)^2 dv_g &= \int_M ( Q_g -\overline{Q}_g ) \Delta_g u\ dv_g  = - \int_M \langle\nabla Q_g, \nabla u \rangle dv_g = - \frac{4n}{n-4} \int_M \langle div_g \overset{\circ}{J}_g, \nabla u \rangle,
\end{align*}
where for the last step we use the fact $$div_g \overset{\circ}{J}_g = div_g \left( J_g - \frac{1}{n} Q_g g\right) = \frac{1}{4} dQ_g - \frac{1}{n} dQ_g = \frac{n-4}{4n} dQ_g.$$

Integrating by parts,
\begin{align*}
 - \frac{4n}{n-4} \int_M \langle div_g \overset{\circ}{J}_g, \nabla u \rangle dv_g
 =& \frac{4n}{n-4} \int_M \langle \overset{\circ}{J}_g, \nabla^2 u \rangle dv_g\\
 =& \frac{4n}{n-4} \int_M \langle \overset{\circ}{J}_g, \nabla^2 u - \frac{1}{n} g\Delta_g u \rangle dv_g \\
 \leq &\frac{4n}{n-4} \left( \int_M |\overset{\circ}{J}_g |^2dv_g \right)^{\frac{1}{2}} \left( \int_M \left|\nabla^2 u - \frac{1}{n} g\Delta_g u \right|^2 dv_g \right)^{\frac{1}{2}}\\
 =& \frac{4n}{n-4} \left( \int_M |\overset{\circ}{J}_g |^2dv_g \right)^{\frac{1}{2}} \left( \int_M |\nabla^2 u |^2 - \frac{1}{n} (\Delta_g u)^2 dv_g \right)^{\frac{1}{2}}.
\end{align*}

From \emph{Bochner formula} and the assumption $Ric_g > 0$,
\begin{align*}
\int_M |\nabla^2 u|^2 dv_g = \int_M (\Delta_g u)^2 dv_g - \int_M Ric_g (\nabla u, \nabla u) dv_g \leq \int_M (\Delta_g u)^2 dv_g.
\end{align*}

Thus,
\begin{align*}
\int_M (Q_g - \overline{Q}_g)^2 dv_g &\leq \frac{4n}{n-4} \left( \int_M |\overset{\circ}{J}_g |^2dv_g \right)^{\frac{1}{2}} \left(  \frac{n - 1}{n} (\Delta_g u)^2 dv_g \right)^{\frac{1}{2}} \\
&= \frac{4n}{n-4} \left( \int_M |\overset{\circ}{J}_g |^2dv_g \right)^{\frac{1}{2}} \left(  \frac{n - 1}{n} (Q_g - \overline{Q}_g)^2 dv_g \right)^{\frac{1}{2}}.
\end{align*}

That is,
\begin{align*}
\int_M (Q_g - \overline{Q}_g)^2 dv_g &\leq \frac{16n(n-1)}{(n-4)^2} \int_M |\overset{\circ}{J}_g |^2dv_g .
\end{align*}

Now we consider the equality case.

If $g$ is $J$-Einstein, then $Q_g$ is a constant by \emph{Schur Lemma} (Theorem \ref{thm:Schur_Lemma}). Thus both sides of inequality (\ref{ineq:almost_Schur}) vanish and equality is achieved.

On the contrary, assume in (\ref{ineq:almost_Schur}) equality is achieved:
\begin{align*}
\int_M (Q_g - \overline{Q}_g)^2 dv_g = \frac{16n(n-1)}{(n-4)^2}\int_M |\overset{\circ}{J}_g|^2 dv_g.
\end{align*}

Then in particular we have 
\begin{align*}
Ric(\nabla u, \nabla u) = 0,
\end{align*}
which implies that $\nabla u = 0$ and hence $u$ is a constant on $M$, since we assume $Ric_g > 0$.

Thus $Q \equiv \overline Q$ on $M$ and 
\begin{align*}
\int_M |\overset{\circ}{J}_g|^2 dv_g = \frac{(n-4)^2}{16n(n-1)}\int_M (Q_g - \overline{Q}_g)^2 dv_g = 0.
\end{align*}

Therefore, $\overset{\circ}{J}_g \equiv 0$ on $M$, \emph{i.e.} $(M, g)$ is $J$-Einstein.
\end{proof}

\begin{remark}
By assuming $$Ric \geq - (n-1) K g$$ for some constant $K \geq 0$ and following the proof in \cite{Cheng}, the inequality (\ref{ineq:almost_Schur}) can be improved to 
\begin{align}
\int_M (Q_g - \overline{Q}_g)^2 dv_g \leq \frac{16n(n-1)}{(n-4)^2} \left( 1 + \frac{nK}{\lambda_1}\right) \int_M |\overset{\circ}{J}_g|^2 dv_g,
\end{align}
where $\lambda_1 > 0$ is the first non-zero eigenvalue of $(-\Delta_g)$.
\end{remark}

Now we can derive an equivalent form of inequality (\ref{ineq:almost_Schur}):
\begin{corollary}
For $n\neq 4$, let $(M^n,g)$ be an $n$-dimensional closed Riemannian manifold with positive Ricci curvature. Then
\begin{align}
( Vol_g(M) )^{-\frac{n-8}{n}}\int_M \sigma_2^J(g) dv_g \leq \frac{n-1}{2n} Y^2_Q(g).
\end{align}
Moreover, the equality holds if and only if $(M^n,g)$ is $J$-Einstein.

\end{corollary}

\begin{proof}
Note that
\begin{align*}
\sigma_1^J (g) = tr_g S_J = \frac{1}{4(n-1)} Q_g
\end{align*}
and 
\begin{align*}
\sigma_2^J (g) = \frac{1}{2} \left( (\sigma_1^J)^2 - |S_J|^2 \right) = \frac{n-1}{2n}(\sigma_1^J)^2 - \frac{1}{2(n-4)^2} |\overset{\circ}{J}_g|^2,
\end{align*}
where we use the fact $$|S_J|^2 = \left|\overset{\circ}{S}_J + \frac{1}{n} (tr_g S_J) g\right|^2 = \left|\frac{1}{n-4}\overset{\circ}{J}_g + \frac{1}{n} (\sigma_1^J) g\right|^2 = \frac{1}{(n-4)^2} |\overset{\circ}{J}_g|^2 + \frac{1}{n} (\sigma_1^J)^2.$$

By substituting these terms in the inequality (\ref{ineq:almost_Schur}), we get
\begin{align*}
\left(\int_M \sigma_1^J(g) dv_g \right)^2 \geq \frac{2n}{n-1} Vol_g(M) \int_M \sigma_2^J(g) dv_g .
\end{align*}

Therefore,
\begin{align*}
\int_M \sigma_2^J(g) dv_g &\leq \frac{n-1}{2n} ( Vol_g(M) )^{-1}\left(\int_M \sigma_1^J(g) dv_g \right)^2\\
&= \frac{n-1}{2n} ( Vol_g(M) )^{\frac{n-8}{n}}\left(\frac{\int_M \sigma_1^J(g) dv_g }{( Vol_g(M) )^{\frac{n-4}{n}}}\right)^2\\
&= \frac{n-1}{2n} ( Vol_g(M) )^{\frac{n-8}{n}}Y^2_Q(g).
\end{align*}

\end{proof}

\begin{remark}
Note that, the \emph{$Q$-Yamabe quotient}
\begin{align*}
Y_Q(g):=\frac{\int_M \sigma_1^J(g) dv_g }{( Vol_g(M) )^{\frac{n-4}{n}}}
\end{align*}
is scaling invariant and in particular, when $n=8$,
\begin{align*}
\int_M \sigma_2^J(g) dv_g &\leq \frac{7}{16} Y^2_Q(g),
\end{align*}
provided that $Ric_g > 0$, where the equality holds if and only if $(M,g)$ is $J$-Einstein.
\end{remark}

\bibliographystyle{amsplain}

\end{document}